\documentclass[12pt,a4paper]{article}
\usepackage{latexsym,amssymb,amsfonts,amsmath,amsthm,nccmath,float,enumitem}
\usepackage[hidelinks]{hyperref}
\usepackage[margin=2cm]{geometry}
\usepackage{multirow}

\newtheorem{theorem}{Theorem}
\newtheorem{corollary}[theorem]{Corollary}
\newtheorem{lemma}[theorem]{Lemma}

\usepackage[dvipsnames]{xcolor}
\theoremstyle{definition}

\newtheorem{remark}[theorem]{Remark}

\usepackage{multirow}
\def \deg {{\rm deg}}
\def \wt {{\rm wt}}
\def \rank {{\rm rank}}

\def \leq {\leqslant}
\def \geq {\geqslant}
\def \B {\mathcal{B}}

\def \mod#1{{\:({\rm mod}\ #1)}}

\let\oldproofname=\proofname
\renewcommand{\proofname}{\rm\bf{\oldproofname}}

\title{\bf New lower bounds for $t$-coverings}

\author{\hspace{-0.7cm}
\begin{minipage}[c]{0.5\textwidth}
\begin{center}
Daniel Horsley\\
School of Mathematical Sciences\\
Monash University\\
Vic 3800, Australia\\
{\tt danhorsley@gmail.com}
\end{center}
\end{minipage}
\begin{minipage}[c]{0.5\textwidth}
\begin{center}
Rakhi Singh\\
IITB-Monash Research Academy \\
Indian Institute of Technology Bombay \\
Mumbai 400076, India\\
{\tt agrakhi@gmail.com}
\end{center}
\end{minipage}
}

\date{}

\begin{document}
\sloppy
\maketitle
\def\baselinestretch{1.2}\small\normalsize

\begin{abstract}
Fisher proved in 1940 that any $2$-$(v,k,\lambda)$ design with $v>k$ has at least $v$ blocks. In 1975 Ray-Chaudhuri and Wilson generalised this result by showing that every $t$-$(v,k,\lambda)$ design with $v \geq k+\lfloor t/2 \rfloor$ has at least $\binom{v}{\lfloor t/2 \rfloor}$ blocks. By combining methods used by Bose and Wilson in proofs of these results, we obtain new lower bounds on the size of $t$-$(v,k,\lambda)$ coverings. Our results generalise lower bounds on the size of $2$-$(v,k,\lambda)$ coverings recently obtained by the first author.
\end{abstract}

\section{Introduction}

For our purposes, an \emph{incidence structure} is a pair $(V,\mathcal{B})$ where $V$ is a set of \emph{points} and $\mathcal{B}$ is a multiset of subsets of $V$ called \emph{blocks}. For positive integers $t$, $v$, $k$ and $\lambda$ with $t \leq k \leq v$, a \emph{$t$-$(v,k,\lambda)$ covering} is an incidence structure $(V,\mathcal{B})$ such that $|V|=v$, $|B|=k$ for all $B \in \mathcal{B}$, and each $t$-subset of $V$ is contained in at least $\lambda$ blocks in $\mathcal{B}$. If each $t$-subset of $V$ is contained in exactly $\lambda$ blocks in $\mathcal{B}$, then $(V,\mathcal{B})$ is a \emph{$t$-$(v,k,\lambda)$ design}. For an incidence structure $(V,\mathcal{B})$ and a subset $X \subseteq V$, define $b(X)$ to be the number of blocks in $\mathcal{B}$ that contain $X$. Coverings were introduced for $t=2$ by Erd{\H{o}}s and R{\'e}nyi \cite{ErdRen} in 1956 and then generalised to arbitrary $t$ by Erd{\H{o}}s and Hanani \cite{ErdHan} in 1963.

Usually we are interested in finding coverings with as few blocks as possible. The \emph{covering number} $C_{\lambda}(v,k,t)$ is the minimum number of blocks in any $t$-$(v,k,\lambda)$ covering. When $\lambda=1$ we omit the subscript. It is convenient to set $C_{\lambda}(v,k,0)=\lambda$ for all $v$, $k$ and $\lambda$. In \cite{Rod} R\"{o}dl introduced the famous \emph{nibble} method to show that $C(v,k,t) \sim \binom{v}{t}/\binom{k}{t}$ as $v \rightarrow \infty$.

Observe that if $(V,\mathcal{B})$ is a $t$-$(v,k,\lambda)$ covering and $X$ is a subset of $V$ with $|X| \leq t$, then $(V \setminus X,\mathcal{B}')$, where $\mathcal{B}'=\{B \setminus X:B \in \mathcal{B}, X \subseteq B\}$, is a $(t-|X|)$-$(v-|X|,k-|X|,\lambda)$ covering and hence
\begin{equation}\label{equation:Subset}
b(X) \geq C_{\lambda}(v-|X|,k-|X|,t-|X|).
\end{equation}
Using this fact with $|X|=1$ and some simple counting gives
\begin{equation}\label{equation:Sch}
C_{\lambda}(v,k,t) \geq \left\lceil\mfrac{v}{k}\:C_{\lambda}(v-1,k-1,t-1)\right\rceil.
\end{equation}
Iterating this inequality yields the \emph{Sch{\"o}nheim bound} \cite{Sch} which states that $C_{\lambda}(v,k,t) \geq L_{\lambda}(v,k,t)$ where
\[L_{\lambda}(v,k,t) = \left\lceil\mfrac{v}{k} \left\lceil\mfrac{v-1}{k-1} \cdots \left\lceil\mfrac{v-t+2}{k-t+2} \left\lceil\mfrac{\lambda(v-t+1)}{k-t+1} \right\rceil \right\rceil \cdots \right\rceil \right\rceil.\]

Furthermore, Mills and Mullin \cite{MilMul} have shown  that if $vC_{\lambda}(v-1,k-1,t-1) \not\equiv 0 \mod{k}$ and $C_{\lambda}(v-1,k-1,t-1) = (\binom{v-1}{r-1}/\binom{k-1}{r-1})C_{\lambda}(v-r,k-r,t-r)$ for some $r \in \{2,\ldots,t\}$, then
\begin{equation}\label{equation:MilMul}
C_{\lambda}(v,k,t) \geq \left\lceil\mfrac{v}{k}(C_{\lambda}(v-1,k-1,t-1)+r)\right\rceil.
\end{equation}
This result is easiest to apply in the case $r=t=2$, when it states that if $\lambda(v-1) \equiv 0 \mod{k-1}$ and $\lambda v(v-1) \equiv 1 \mod{k}$, then $C_{\lambda}(v,k,t) \geq L_{\lambda}(v,k,t)+1$. A result of Keevash \cite[Theorem 6.5]{Kee} implies that, for a fixed $t$, $k$ and $\lambda$ and for all sufficiently large $v$, $C_{\lambda}(v,k,t)=h_{\lambda}(v,k,t)/\binom{k}{t}$ where $h_{\lambda}(v,k,t)$ is the size of a smallest $t$-$(v,t,\lambda)$ covering $(V,\mathcal{B})$ with the property that $\binom{k-|X|}{t-|X|}$ divides $b(X)$ for each subset $X$ of $V$ with $|X| \leq t$. In the case $t=2$, this establishes that the Sch{\"o}nheim bound with the Mills and Mullin improvement is tight for all sufficiently large $v$. Glock et al. \cite{GloKuhLoOst} have recently extended Keevash's main result.

Our interest here is principally in establishing lower bounds for covering numbers $C_{\lambda}(v,k,t)$ when $k$ is a significant fraction of $v$. Exact values for $C_{\lambda}(v,k,t)$ have been determined when $(k,t) \in \{(3,2),(4,2)\}$, when $(t,\lambda)=(2,1)$ and $v \leq \frac{13}{4}k$, and for most cases when $(t,\lambda)=(3,1)$ and $v \leq \frac{8}{5}k$ (see \cite{GorSti}). In the case $t=2$, a number of results have been proved which improve on the Sch{\"o}nheim bound in various cases where $k$ is a significant fraction of $v$ \cite{BluGreHei,BosCon,BryBucHorMaeSch,Fur,TodFP,TodLB}. A number of other lower bounds for specific parameter sets, which have been mostly obtained by computer searches, are available in literature (see \cite{Gor,GorSti}). For surveys on coverings see \cite{GorSti,MilMul}. Gordon maintains a repository for small coverings \cite{Gor}.

Fisher's inequality \cite{Fis} famously states that every $2$-$(v,k,\lambda)$ design with $v>k$ has at least $v$ blocks. Ray-Chaudhuri and Wilson \cite{raywil75} generalised this result to higher $t$ by showing that every $t$-$(v,k,\lambda)$ design with $v \geq k+s$ has at least $\binom{v}{s}$ blocks for any positive integer $s \leq \lfloor\frac{t}{2}\rfloor$. Subsequently Wilson \cite{wil82} gave an alternate proof of this generalised result using so-called \emph{higher incidence matrices}. In this paper we demonstrate how an approach based on these matrices can be used to obtain improved lower bounds on covering numbers $C_\lambda(v,k,t)$. Our results generalise both the Ray-Chaudhuri and Wilson result of \cite{raywil75} and the more recent results of \cite{Hor} which established lower bounds for $C_\lambda(v,k,2)$.

To avoid triviality, we often consider only $t$-$(v,k,\lambda)$ coverings with $2 < k < v$. The bounds we prove in this paper apply to covering numbers $C_\lambda(v,k,t)$ for arbitrary $\lambda$. However in our discussions, as in most of the literature concerning coverings with $t \geq 3$, we will concentrate on the case $\lambda=1$. The methods in this paper should also be applicable to packings, but we do not pursue this here.

In the next section we discuss our proof strategy and prove some preliminary results. In Sections~\ref{basicSec},~\ref{genSec} and~\ref{impSec} we then prove and discuss bounds that generalise Theorems 1, 11 and 14 of \cite{Hor} respectively. The results in Sections~\ref{genSec} and~\ref{impSec} make use of a result of Caro and Tuza \cite{CarTuz} which guarantees an $m$-independent set of a certain size in a multigraph with a specified degree sequence. In Section~\ref{Section:exact} we exhibit infinite families of parameter sets $t$-$(v,k,\lambda)$ for which our results improve on the best bounds previously known.

\section{Strategy and preliminary results}\label{Section:strategy}

To prove our results we will combine ideas from \cite{Hor} with those from a proof by Wilson \cite{wil82} of the generalisation of Fisher's inequality to higher $t$. The methods in \cite{Hor} were, in turn, inspired by a proof by Bose \cite{Bos} of Fisher's inequality. Following \cite{wil82}, we make use of higher incidence matrices. For a nonnegative integer $s$, the \emph{$s$-incidence matrix} of an incidence structure $(V,\mathcal{B})$ is the matrix whose rows are indexed by the $s$-subsets of $V$, whose columns are indexed by the blocks in $\mathcal{B}$, and where the entry in row $X$ and column $B$ is 1 if $X \subseteq B$ and 0 otherwise. For a set $V$ and a nonnegative integer $i$, let $\binom{V}{i}$ denote the set of all $i$-subsets of $V$.

We will make use of standard facts about positive definite matrices (see \cite[\S9.4]{Hog}). If $A$ is a square matrix whose rows and columns are indexed by the elements of a set $Z$, then a \emph{principal submatrix} of $A$ is a square submatrix whose rows and columns are both indexed by the same subset $Z'$ of $Z$. We say a real matrix is \emph{diagonally dominant} if, in each of its rows, the magnitude of the diagonal entry is strictly greater than the sum of the magnitudes of the other entries in that row. It follows easily from the well-known Gershgorin circle theorem (see \cite[p16-6]{Hog}) that real diagonally dominant matrices are positive definite. Our bounds rest on the following simple observations.

\begin{lemma}\label{Lemma:basicFacts}
Let $(V,\mathcal{B})$ be an incidence structure and let $A$ be the $s$-incidence matrix of $(V,\mathcal{B})$ for some positive integer $s$. Then
\begin{itemize}[nosep]
    \item[(i)]
$AA^T$ is the symmetric matrix whose row and columns are indexed by $\binom{V}{s}$ and where the entry in row $X$ and column $Y$ is $b(X \cup Y)$; and
    \item[(ii)]
$|\mathcal{B}| \geq \rank(AA^T)$.
\end{itemize}
\end{lemma}

\begin{proof}
Part (i) follows from the definition of matrix multiplication. Because $A$ has only $|\mathcal{B}|$ columns, $\rank(A) \leq |\mathcal{B}|$. Thus $|\mathcal{B}| \geq \rank(A) \geq \rank(AA^T)$, proving part (ii).
\end{proof}

By Lemma~\ref{Lemma:basicFacts} we can bound the number of blocks in a covering by bounding $\rank(AA^T)$. Our strategy to bound this rank is as follows. We first write $AA^T=P+M$ where $P$ is positive semidefinite. We then find a diagonally dominant, and hence positive definite, principal submatrix $M'$ of $M$. Because every principal submatrix of $P$ is positive semidefinite, the submatrix of $AA^T$ with row and column indices corresponding to those of $M'$ is positive definite and hence full rank. Thus the rank of $AA^T$ is at least the order of $M'$.

We choose $P$ so that the entry in row $X$ and column $Y$ for $X \neq Y$ is $b_{|X \cup Y|}$, where $b_{s+1},\ldots,b_{2s}$ are positive integers chosen so that each $i$-subset of $V$ is in at least $b_i$ blocks in $\mathcal{B}$ for $i \in \{s+1,\ldots,2s\}$. The entries on the lead diagonal of $P$ are chosen to be small as possible, given that $P$ must be positive semidefinite. We establish that $P$ is indeed positive semidefinite using an approach from \cite{wil82} in which $P$ is written as a nonnegative linear combination of Gram matrices.

We will require the following simple identity for binomial coefficients.

\begin{lemma}\label{Lemma:multinomIdent}
Let $i$ and $\ell$ be nonegative integers with $i \leq \ell$. Then
\[\medop\sum_{j=i}^{\ell} (-1)^{i+j}\tbinom{\ell}{j}\tbinom{j}{i}=
\left\{
  \begin{array}{ll}
    0, & \hbox{if $i<\ell$;} \\
    1, & \hbox{if $i=\ell$.}
  \end{array}
\right.\]
\end{lemma}

\begin{proof}
The multinomial theorem implies that the coefficient of $x^i$ in the expansion of $(x-1+1)^\ell$ is
\[\medop\sum_{j'=0}^{\ell-i} \tbinom{\ell}{i+j'}\tbinom{i+j'}{i}(-1)^{j'}=\medop\sum_{j=i}^\ell (-1)^{i+j}\tbinom{\ell}{j}\tbinom{j}{i},\]
where the equality is obtained by substituting $j=i+j'$.
So because $(x-1+1)^\ell=x^\ell$, the result now follows by equating the coefficients of $x^i$.
\end{proof}

The next lemma establishes that if $A$ is the higher incidence matrix of a $t$-$(v,k,\lambda)$ covering, then $AA^T$ has a specific form that we can exploit. Subsequent results in this paper will often explicitly assume the hypotheses of Lemma~\ref{Lemma:setUp} and use its notation.

\begin{lemma}\label{Lemma:setUp}
Let $t$, $v$, $k$, $\lambda$ and $s$ be positive integers such that $t < k < v$ and $s \leq \lfloor\frac{t}{2}\rfloor$. Let $b_{2s},b_{2s-1},\ldots,b_s$ be positive integers such that
\begin{itemize}
    \item[(i)]
$L_\lambda(v-2s,k-2s,t-2s) \leq b_{2s} \leq C_\lambda(v-2s,k-2s,t-2s)$;
    \item[(ii)]
$\lceil \frac{v-i}{k-i} b_{i+1} \rceil \leq b_i \leq C_\lambda(v-i,k-i,t-i)$ for $i = 2s-1,2s-2,\ldots,s$; and
    \item[(iii)]
$a_j \geq 0$ for $j \in \{0,\ldots,s\}$, where $a_j=\sum_{i=0}^j (-1)^{i+j}\binom{j}{i}b_{2s-i}$.
\end{itemize}
If $(V,\mathcal{B})$ is a $t$-$(v,k,\lambda)$ covering and $A$ is the $s$-incidence matrix of $(V,\mathcal{B})$, then $b(Z) \geq b_{|Z|}$ for any $Z \subseteq V$ with $|Z| \in \{s,\ldots,2s\}$ and $AA^T = P+M$ for matrices $P=(p_{XY})$ and $M=(m_{XY})$ such that
\[
p_{XY}=
\left\{
  \begin{array}{ll}
    b_{|X \cup Y|} & \hbox{if $X \neq Y$} \\
    b_s-a_s & \hbox{if $X=Y$}
  \end{array}
\right. \qquad
m_{XY}=
\left\{
  \begin{array}{ll}
    b(X \cup Y) - b_{|X \cup Y|} & \hbox{if $X \neq Y$} \\
    a_s+b(X)-b_s & \hbox{if $X=Y$}
  \end{array}
\right..\]
Furthermore, the following hold.
\begin{itemize}
    \item[(a)]
$P=\sum_{j=0}^{s-1} a_j Q_j^TQ_j$, where $Q_j$ is the $j$-incidence matrix of the incidence structure $(V,\binom{V}{s})$. Hence $P$ is positive semidefinite.
    \item[(b)]
For any $X \in \binom{V}{s}$,
\[\medop \sum_{Y \in \binom{V}{s}\setminus \{X\}} m_{XY} = \medop \sum_{Y \in \binom{V}{s} \setminus \{X\}} \left(b(X \cup Y)-b_{|X \cup Y|}\right)=\left(b(X)-b_s\right)\left(\tbinom{k}{s}-1\right)+d\]
where $d=b_s(\tbinom{k}{s}-1)-\sum_{i=0}^{s-1}\tbinom{s}{i}\tbinom{v-s}{s-i}b_{2s-i}$ is a nonnegative integer.
\end{itemize}
\end{lemma}

\begin{proof}
Let $Z \subseteq V$ with $|Z| \in \{s,\ldots,2s\}$. That $b(Z) \geq b_{|Z|}$ follows because $b_{|Z|} \leq C_\lambda(v-|Z|,k-|Z|,t-|Z|)$ by (i) and (ii) and $C_\lambda(v-|Z|,k-|Z|,t-|Z|) \leq b(Z)$ by \eqref{equation:Subset}. That $AA^T = P+M$ follows immediately from Lemma~\ref{Lemma:basicFacts} (i) and the definitions of $P$ and $M$. Let $\mathcal{V}=\binom{V}{s}$ and $\mathcal{V}_0=\{X \in \mathcal{V}:b(X)=b_s\}$.

We prove (a). Observe that for $j \in \{0,\ldots,s\}$, $Q_j^TQ_j$ is the matrix whose rows and columns are indexed by $\binom{V}{s}$ and whose $(X,Y)$ entry is $\binom{|X \cap Y|}{j}$ for all $X,Y \in \binom{V}{s}$. In particular, $Q_s^TQ_s=I$. Let
\[Q'=\medop\sum_{j=0}^{s}a_j Q_j^TQ_j=a_sI+\medop\sum_{j=0}^{s-1}a_jQ_j^TQ_j.\]
It suffices to show that $Q'=a_sI+P$.

Let $X,Y \in \binom{V}{s}$, let $\ell = |X \cap Y|$, and note that $\ell \leq s$. For $j \in \{0,\ldots,s\}$, the $(X,Y)$ entry of $Q_j^TQ_j$ is $\binom{\ell}{j}$. Thus the $(X,Y)$ entry of $Q'$ is
\[\medop\sum_{j=0}^{s} a_j\tbinom{\ell}{j} = \medop\sum_{j=0}^{s} \medop\sum_{i=0}^{j} (-1)^{i+j}\tbinom{\ell}{j}\tbinom{j}{i}b_{2s-i} = \medop\sum_{i=0}^{s} \medop\sum_{j=i}^{s} (-1)^{i+j}\tbinom{\ell}{j}\tbinom{j}{i}b_{2s-i}.\]
So it follows from Lemma~\ref{Lemma:multinomIdent} that the $(X,Y)$ entry of $Q'$ is $b_{2s-\ell}=b_{|X \cup Y|}$. Thus $Q'=a_sI+P$.

Now we prove (b).  For each $X \in \mathcal{V}$,
\[\medop\sum_{Y \in \mathcal{V}\setminus\{X\}}b(X \cup Y)=b(X)\left(\tbinom{k}{s}-1\right)\]
because each block that contains $X$ contributes $\binom{k}{s}-1$ to this sum. Also for each $X \in \mathcal{V}$,
\[\medop\sum_{Y \in \mathcal{V}\setminus\{X\}}b_{|X \cup Y|}=\medop\sum_{i=0}^{s-1}\tbinom{s}{i}\tbinom{v-s}{s-i}b_{2s-i}\]
because, for each $i \in \{0,\ldots,s-1\}$, $|\{Y:|X \cap Y|=i\}|=\binom{s}{i}\binom{v-s}{s-i}$. Together, these facts imply that (b) holds provided $d$ is nonnegative. By (ii), $b_{i+1} \leq \frac{k-i}{v-i}b_i$ for $i = 2s-1,2s-2,\ldots,s$ and so it can be seen that $b_{2s-i} \leq (\binom{k-s}{s-i}/\binom{v-s}{s-i})b_s$ for $i=s-1,s-2,\ldots,0$. Thus, \[\medop\sum_{i=0}^{s-1}\tbinom{s}{i}\tbinom{v-s}{s-i}b_{2s-i} \leq b_s\medop\sum_{i=0}^{s-1}\tbinom{s}{i}\tbinom{k-s}{s-i} =  b_s(\tbinom{k}{s}-1),\]
and it follows that $d \geq 0$.
\end{proof}

\begin{remark}\label{Remark:bigv}
In many cases condition (ii) of Lemma~\ref{Lemma:setUp} implies condition (iii). Specifically, we claim that if condition (ii) is satisfied then $a_j \geq 0$ for $j \in \{0,\ldots,\min(\lfloor\frac{v}{k}\rfloor,{s})\}$. This means that we can ignore condition (iii) whenever $v \geq sk$. Certainly, $a_0=b_{2s}\geq1$. To see that the rest of our claim is true, fix $j \in \{1,\ldots,\min(\lfloor\frac{v}{k}\rfloor,{s})\}$, and let $\delta=2$ if $j$ is even and $\delta=1$ if $j$ is odd. Then, pairing consecutive terms in the definition of $a_j$, we see that
\[a_j\geq\medop\sum_{i\in\{\delta,\delta+2,\ldots,j\}}\left(\tbinom{j}{i}b_{2s-i}-\tbinom{j}{i-1}b_{2s-i+1}\right).\]
For $i\in\{\delta,\delta+2,\ldots,j\}$, using condition (ii),
\[\tbinom{j}{i}=\tfrac{j-i+1}{i}\tbinom{j}{i-1} \geq \tfrac{1}{j}\tbinom{j}{i-1} \quad\mbox{ and }\quad b_{2s-i} \geq \left\lceil \tfrac{v-2s+i}{k-2s+i} b_{2s-i+1} \right\rceil \geq \tfrac{v}{k} b_{2s-i+1} \geq jb_{2s-i+1},\]
and hence $\tbinom{j}{i}b_{2s-i} \geq \tbinom{j}{i-1}b_{2s-i+1}$. Thus $a_j \geq 0$.
\end{remark}

It follows from Lemma~\ref{Lemma:setUp}(a) that the diagonal entries $b_s-a_s$ of $P$ are at least $a_0=b_{2s}>0$. Hence $b_s>a_s$. This fact will be used several times in later sections. We are now ready to prove Lemma~\ref{Lemma:main}, which forms the basis of all the lower bounds that we establish in this paper.

\begin{lemma}\label{Lemma:main}
Suppose the hypotheses of Lemma~\ref{Lemma:setUp} hold. If there is a subset $\mathcal{S}$ of $\binom{V}{s}$ such that, for each $X \in \mathcal{S}$,
\[\medop \sum_{Y \in \mathcal{S} \setminus \{X\}} \left(b(X \cup Y)-b_{|X \cup Y|}\right) < a_s+b(X)-b_s,\]
then $|\mathcal{B}| \geq |\mathcal{S}|$.
\end{lemma}

\begin{proof}
By Lemma~\ref{Lemma:basicFacts} (ii), it suffices to show that the principal submatrix of $AA^T$ whose rows and columns are indexed by $\mathcal{S}$ is positive definite and hence full rank.

By Lemma~\ref{Lemma:setUp}, $AA^T$ can be written as the sum of a positive semidefinite matrix $P$ and a matrix $M$ whose $(X,Y)$ entry is the nonnegative integer $b(X \cup Y)-b_{|X \cup Y|}$ for all distinct $X,Y \in \binom{V}{s}$ and whose $(X,X)$ entry is the nonnegative integer $a_s+b(X)-b_{s}$ for all $X \in \binom{V}{s}$. Because every principal submatrix of $P$ is positive semidefinite, it in fact suffices to show that the principal submatrix $M'$ of $M$ whose rows and columns are indexed by $\mathcal{S}$ is positive definite. Given the hypothesis of the lemma that
\[\medop \sum_{Y \in \mathcal{S} \setminus \{X\}} \left(b(X \cup Y)-b_{|X \cup Y|}\right) < a_s+b(X)-b_s,\]
$M'$ is diagonally dominant and hence it is positive definite by the Gershgorin circle theorem (see \cite[p.16-6]{Hog}).
\end{proof}

\section{Basic bound}\label{basicSec}

Here we use Lemma~\ref{Lemma:main} to prove the simplest and most easily stated of our results, and then discuss when it can be usefully applied.

\begin{theorem}\label{Theorem:main}
Suppose the hypotheses of Lemma~\ref{Lemma:setUp} hold and that $d < a_s$. Then
\[C_{\lambda}(v,k,t) \geq \left\lceil\mfrac{\binom{v}{s}(b_s+1)}{\binom{k}{s}+1}\right\rceil.\]
\end{theorem}

\begin{proof}
Let $(V,\mathcal{B})$ be a $t$-$(v,k,\lambda)$ covering. Let $\mathcal{V}=\binom{V}{s}$ and $\mathcal{V}_0=\{X \in \mathcal{V}:b(X)=b_s\}$. Because $d < a_s$, it follows from Lemma~\ref{Lemma:setUp}(b) that we can apply Lemma~\ref{Lemma:main} with $\mathcal{S}=\mathcal{V}_0$ and hence conclude that $|\mathcal{B}| \geq |\mathcal{V}_0|$.

Since each block in $\mathcal{B}$ covers $\binom{k}{s}$ sets in $\mathcal{V}$, we have that $\sum_{X \in \mathcal{V}}b(X)=|\mathcal{B}|\binom{k}{s}$. Thus
\[|\{X \in \mathcal{V}:b(X)>b_s\}| \leq |\mathcal{B}|\tbinom{k}{s}-\tbinom{v}{s}b_s\]
because $b(X) \geq b_s$ for each $X \in \mathcal{V}$. It follows that $|\mathcal{B}| \geq |\mathcal{V}_0| \geq \binom{v}{s}-(|\mathcal{B}|\binom{k}{s}-\binom{v}{s}b_s)$. A simple calculation now establishes that
\[|\mathcal{B}| \geq \mfrac{\binom{v}{s}(b_s+1)}{\binom{k}{s}+1}.\qedhere\]
\end{proof}

It is useless to apply Theorem~\ref{Theorem:main} with $b_s$ chosen to be less than the best known lower bound for
$C_\lambda(v-s,k-s,t-s)$, because the bound of Theorem~\ref{Theorem:main} is always inferior to the bound given by $s$ iterated applications of \eqref{equation:Sch} to $b_s+1$ (note this latter bound is at least $\lceil b_s\binom{v}{s}/\binom{k}{s} \rceil$). Furthermore, from the definitions of $d$ and $a_s$ we have that
\begin{equation}\label{equation:altForm}
a_s-d=\left(\medop\sum_{i=0}^{s-1}\tbinom{s}{i}(\tbinom{v-s}{s-i}+(-1)^{s-i})b_{2s-i}\right)-\left(\tbinom{k}{s}-2\right)b_s,
\end{equation}
which is increasing in $b_{2s-i}$ for each $i \in \{0,\ldots,s-1\}$.
Thus, in the absence of condition (iii) of Lemma~\ref{Lemma:setUp}, it can be seen that when attempting to apply  Theorem~\ref{Theorem:main} we only need consider choosing $b_i$ to be the best known lower bound on $C_\lambda(v-i,k-i,t-i)$ for $i \in \{s,\ldots,2s\}$. Throughout the rest of the paper, we shall refer to this as the natural choice for the $b_i$. Condition (iii) complicates the picture somewhat, but in view of Remark~\ref{Remark:bigv} this is only of concern when $v \leq (s-1)k$ (note that $a_s > d \geq 0$ by our hypotheses and Lemma~\ref{Lemma:setUp}). In many cases the best known lower bounds are all given by the Sch{\"o}nheim bound and in these cases the natural choice of the $b_i$ amounts to taking $b_i=L_\lambda(v-i,k-i,t-i)$ for $i \in \{s,\ldots,2s\}$.

For each of the subsequent lower bounds we establish in this paper (see Theorems~\ref{Theorem:dBig} and \ref{Theorem:smallDImprovements}), we will also show that we only need consider the natural choice for $b_s$. With this choice fixed, the natural choice for the remaining $b_i$ will minimise $d$ and maximise $a_s-d$, by the definition of $d$ and \eqref{equation:altForm}. Considering this and Remark~\ref{Remark:bigv}, we believe that taking the natural choice for the $b_i$ in our theorems will almost always produce the best results.

For the Theorem~\ref{Theorem:main} bound to exceed the bound obtained by $s$ iterated applications of \eqref{equation:Sch} to $b_s$, it must be the case that $b_s < \binom{k}{s}$ (again note the latter bound is at least $\lceil b_s\binom{v}{s}/\binom{k}{s} \rceil$). Furthermore, the other lower bounds we establish in this paper will explicitly require $b_s < \binom{k}{s}$.
We have $b_s < \binom{k}{s}$ only when $v < (\frac{k^t}{s!\lambda})^{1/(t-s)}$ because $\binom{k}{s} \leq \frac{k^s}{s!}$ and
$b_s \geq L_\lambda(v-s,k-s,t-s)\geq \lambda\binom{v-s}{t-s}/\binom{k-s}{t-s} > \lambda(\frac{v}{k})^{t-s}$. So none of the lower bounds of this paper are of use when $v \geq (\frac{k^t}{s!\lambda})^{1/(t-s)}$.

Theorem~\ref{Theorem:main} implies Ray-Chaudhuri and Wilson's \cite{raywil75} generalisation of Fisher's inequality. If there exists a $t$-$(v,k,\lambda)$ design $(V,\mathcal{B})$ with $v \geq k+s$ for some positive integer $s \leq \lfloor\frac{t}{2}\rfloor$, then applying Theorem~\ref{Theorem:main} with $b_i=L_\lambda(v-i,k-i,t-i)=\lambda\binom{v-i}{t-i}/\binom{k-i}{t-i}$ for $i \in \{s,\ldots,2s\}$ we have $C_\lambda(v,k,t) \geq \binom{v}{s}(b_s+1)/(\binom{k}{s}+1)$ (the hypotheses are satisfied because $d=0$ and $a_j=\lambda\binom{v-2s}{k-2s+j}/\binom{v-t}{k-t}$ for $j \in \{0,\ldots,s\}$). But, because $(V,\mathcal{B})$ is a design, it has exactly $\binom{v}{s}b_s/\binom{k}{s}$ blocks. So we can conclude that $\binom{v}{s}b_s/\binom{k}{s} \geq \binom{v}{s}(b_s+1)/(\binom{k}{s}+1)$ which implies $b_s \geq \binom{k}{s}$ and hence that $(V,\mathcal{B})$ has at least $\binom{v}{s}$ blocks.

\section{Infinite families of improvements}\label{Section:exact}

In this section we first give, in Lemma~\ref{Lemma:infFam}, an infinite family of parameter sets for which applying Theorem~\ref{Theorem:main} with $s=2$ yields an improvement over the Sch{\"o}nheim bound. Then we exhibit, in Theorem~\ref{ExactTh2}, an infinite family of parameter sets for which applying Theorem~\ref{Theorem:main} with $s=1$ establishes exact covering numbers.  In this section we will often use the simple observation that, for given $t$, $k$ and $\lambda$, $C_{\lambda}(v,k,t) \leq C_{\lambda}(v',k,t)$ when $v \leq v'$.
	
\begin{lemma}\label{Lemma:infFam}
Let $m \geq 6$ be an integer, and let $v=m^2(m-2)+4$ and $k=m(m-1)+2$. An application of Theorem~\ref{Theorem:main} with $s=2$ establishes that $C(v,k,5) \geq L(v,k,5)+m(m-4)-10$.
\end{lemma}
	
\begin{proof}
Let $\ell_i=L(v-i,k-i,t-i)$ for $i=4,3,2$. We can successively calculate
\begin{align*}
  \ell_4 &= \lceil\tfrac{v-4}{k-4}\rceil= \lceil m-1+\tfrac{m-2}{m(m-1)-2}\rceil=m \\
  \ell_3 &= \lceil\tfrac{v-3}{k-3} \ell_4\rceil=m(m-1) \\
  \ell_2 &= \lceil\tfrac{v-2}{k-2} \ell_3\rceil=m^2(m-2)+2.
\end{align*}

We will apply Theorem~\ref{Theorem:main} with $s=2$ and $b_i=\ell_i$ for $i=4,3,2$. Routine calculations show that, in the terminology of Lemma~\ref{Lemma:setUp}, $a_0=m$, $a_1=m(m-2)$, $a_2=m^3-4m^2+3m+2$, and $d=0$. Using this, and recalling that $m \geq 6$, it can be seen that the hypotheses of Theorem~\ref{Theorem:main} are satisfied and hence
\[C(v,k,5) \geq \left\lceil\tfrac{v(v-1)}{k(k-1)+2}(\ell_2+1)\right\rceil.\]
This implies that $C(v,k,5) \geq m^5 - 4m^4 + 21m^2 - 14m - 55$.

On the other hand,
\[L(v,k,5)=\lceil\tfrac{v}{k}\lceil\tfrac{v-1}{k-1}\ell_2\rceil\rceil\]
and for $m \geq 14$ we can calculate that this is equal to $m^5 - 4m^4 + 20m^2 - 10m - 45$. Thus it can be seen that the lemma holds for $m \geq 14$, and it is routine to check it holds for $6 \leq m \leq 13$.
\end{proof}

Further routine calculations establish that, for $v$ and $k$ as in Lemma~\ref{Lemma:infFam}, neither the result of Mills and Mullin \cite{MilMul} nor the results of this paper (including those in Sections~\ref{genSec} and \ref{impSec}) give improvements over the Sch{\"o}nheim bound for the parameter sets $C(v-1,k-1,4)$, $C(v-2,k-2,3)$ or $C(v-3,k-3,2)$. We believe that, in general, no bound better than the Sch{\"o}nheim bound was previously known for this family of parameter sets. Since $d=0$ in our application of Theorem~\ref{Theorem:main}, we could make a slight further improvement to this result by instead applying Theorem~\ref{Theorem:smallDImprovements}(a) below.

We now move on to show that Theorem~\ref{Theorem:main} with $s=1$ can be applied to establish that certain coverings constructed from affine planes are optimal, and thus obtain a family of exact covering numbers.

Let $q$ be a prime power. It is well known (see \cite{gordon95}, for example) that if we take $V$ to be the $q^t$ points of the affine geometry $AG(t,q)$ and $\B$ to be the set of its $(t-1)$-flats, then $(V,\B)$ is a $t$-$(q^t,q^{t-1},1)$ covering with $q(\frac{q^t-1}{q-1})$ blocks. Further, it is straightforward to calculate that $L(q^t,q^{t-1},t)=q(\frac{q^t-1}{q-1})$ and hence  $C(q^t,q^{t-1},t) = q(\frac{q^t-1}{q-1})$. The following lemma is based on a well-known ``blow up'' construction for coverings.

\begin{lemma}
	\label{ExactLem}
Let $m$, $t$ and $q$ be positive integers such that $q$ is a prime power. Then $C(v,mq^{t-1},t) \leq q(\frac{q^t-1}{q-1})$ for each $v \leq mq^t$.
\end{lemma}
\begin{proof}
Let $(U,\mathcal{A})$ be the $t$-$(q^t,q^{t-1},1)$ covering with $q(\frac{q^t-1}{q-1})$ blocks obtained from the $(t-1)$-flats of $AG(t,q)$. Let $M$ be a set of $m$ elements, let $V = U \times M$ and let $\mathcal{B}=\{A \times M: A \in \mathcal{A}\}$. Then $(V,\mathcal{B})$ is an $(mq^t,mq^{t-1},1)$-covering with $q(\frac{q^t-1}{q-1})$ blocks. The result now follows because $C(v-1,k,t) \leq C(v,k,t)$ for any parameter set $(v,k,t)$.
\end{proof}

Next we determine the value of the Sch{\"o}nheim bound in the cases we are concerned with.

\begin{lemma}\label{ExactLem2}
Let $v$, $m$, $q$ and $t$ be positive integers such that $q$ is a prime power, $m \geq 2q+2$, $2 \leq t < mq^{t-1}$, and $mq^t-2q+3 \leq v  \leq mq^t$. Let $\ell_t=1$ and let $\ell_i=L(v-i,mq^{t-1}-i,t-i)$ for $i = t-1,t-2,\ldots,0$. Then
\begin{itemize}
    \item[(i)]
$\ell_i=\frac{q^{t-i+1}-1}{q-1}$ for $i = t-1,t-2,\ldots,0$;
    \item[(ii)]
$\ell_1=
\left\{
  \begin{array}{ll}
    \frac{q^{t}-1}{q-1} & \hbox{if $mq^t-q+2 \leq v \leq mq^t$} \\[1mm]
    q(\frac{q^{t-1}-1}{q-1}) & \hbox{if $mq^t-2q+3 \leq v \leq mq^t-q+1$;}
  \end{array}
\right.$
    \item[(iii)]
$\ell_0=
\left\{
  \begin{array}{ll}
    q(\frac{q^{t}-1}{q-1}) & \hbox{if $mq^t-q+2 \leq v \leq mq^t$} \\[1mm]
    q^2(\frac{q^{t-1}-1}{q-1}) & \hbox{if $mq^t-2q+3 \leq v \leq mq^t-q+1$.}
  \end{array}
\right.$
\end{itemize}
\end{lemma}
\begin{proof}
Let $c$ be the integer such that $v=mq^t-q+1+c$. By definition, for $i = t-1,t-2,\ldots,0$,
\begin{equation}\label{elleq}
\ell_i = \left\lceil \mfrac{(mq^t-q+1+c-i)\ell_{i+1}}{mq^{t-1}-i}\right\rceil=  q\ell_{i+1}+ \left\lceil\mfrac{((i-1)(q-1)+c)\ell_{i+1}}{mq^{t-1}-i}\right\rceil.
\end{equation}
Since $c \in \{-q+2,\ldots,q-1\}$, \eqref{elleq} implies that $\ell_i=q\ell_{i+1}+1$ for $i \geq 2$, provided $\ell_{i+1} \leq \frac{mq^{t-1}-i}{i(q-1)}$. Using this fact, it is easy to prove (i) by induction on $i$. In particular, we have $\ell_2=\frac{q^{t-1}-1}{q-1}$, and applying \eqref{elleq} once more establishes (ii). Applying \eqref{elleq} one final time using (ii) and the hypothesis $m \geq 2q+2$ establishes (iii).
\end{proof}	

Together, Lemmas~\ref{ExactLem} and~\ref{ExactLem2} establish the known result that, under the hypotheses of Lemma~\ref{ExactLem2}, $C(v,mq^{t-1},t) = q(\frac{q^t-1}{q-1})$ for $v \in \{mq^t-q+2,\ldots,mq^t\}$. By applying Theorem~\ref{Theorem:main} with $s=1$ we can strengthen this result to cover some cases where $v \leq mq^t-q+1$.

\begin{theorem}\label{ExactTh2}
Let $m$, $q$ and $t$ be positive integers such that $q$ is a prime power, $m \geq 2q+2$ and $2 \leq t < mq^{t-1}$.
Then $C(v,mq^{t-1},t) = q(\frac{q^t-1}{q-1})$ for each integer $v$ such that
\[mq^t-q+1-z \leq v \leq mq^t \quad\mbox{ where }\quad z=\min\left(q-2,\left\lfloor \mfrac{m(q-1)q^{t-1}}{q^t - 1}\right\rfloor - 2q + 1\right).\]
\end{theorem}

\begin{proof}
Note that $z \geq 0$ because $m \geq 2q+2$. Let $v'=mq^t-q+1-z$. It suffices to show that $C(v',mq^{t-1},t) \geq q(\frac{q^t-1}{q-1})$, because then, for each integer $v$ such that $v' \leq v \leq mq^t$, we have
\[q\left(\mfrac{q^t-1}{q-1}\right) \leq C(v',mq^{t-1},t) \leq C(v,mq^{t-1},t) \leq C(mq^t,mq^{t-1},t) \leq q\left(\mfrac{q^t-1}{q-1}\right),\]
where the final inequality follows from Lemma~\ref{ExactLem}.

For $i \in \{0,1,2\}$, let $\ell_i=L(v'-i,mq^{t-1}-i,t-i)$. By Lemma~\ref{ExactLem2}, $\ell_1=q(\frac{q^{t-1}-1}{q-1})$ and $\ell_2=\frac{q^{t-1}-1}{q-1}$. To bound $C(v',mq^{t-1},t)$ below, we will apply Theorem~\ref{Theorem:main} with $s=1$, $b_1=\ell_1$ and $b_2=\ell_2$. Obviously this choice satisfies hypotheses (i) and (ii) of Lemma~\ref{Lemma:setUp}. Because $v' \geq mq^t-2q+3$, a simple calculation establishes that $\ell_1(mq^{t-1}-2) < \ell_2(v'-2)$ and thus $d<a_1$ (because $d \geq 0$, this also implies that $a_1 \geq 0$ and that hypothesis (iii) of Lemma~\ref{Lemma:setUp} holds). So, by Theorem~\ref{Theorem:main}, we have
\[C(v,k,t) \geq \left\lceil\mfrac{v'(\ell_1+1)}{mq^{t-1}+1}\right\rceil=q(\ell_1+1)-\left\lfloor \mfrac{(2q+z-1)(\ell_1+1)}{mq^{t-1}+1} \right\rfloor.\]
A routine calculation shows that the second upper bound on $z$ in our hypotheses is equivalent to $(2q+z-1)(\ell_1+1) \leq mq^{t-1}$ and hence $C(v,k,t) \geq q(\ell_1+1)$. Observing that $q(\ell_1+1)=q(\frac{q^t-1}{q-1})$ completes the proof.
\end{proof}

\begin{corollary}
Let $m$, $q$ and $t$ be positive integers such that $q$ is a prime power, $m \geq 3q$ and $2 \leq t < mq^{t-1}$. Then $C(v,mq^{t-1},t) = q(\frac{q^t-1}{q-1})$ for each integer $v$ such that $mq^t-2q+3 \leq v \leq mq^t$.
\end{corollary}

\begin{proof}
This follows by observing that, in Theorem~\ref{ExactTh2}, $z=q-2$ if $m \geq 3q$.
\end{proof}

\section{Bounds for the case \texorpdfstring{$d \geq a_s$}{d>=as}}\label{genSec}

Using the terminology of Lemma~\ref{Lemma:setUp}, Theorem~\ref{Theorem:main} applies only when $d<a_s$. In this section we will establish a bound that can be applied when $d \geq a_s$. For a multigraph $G$ and a subset $S$ of $V(G)$, let $G[S]$ denote the sub-multigraph of $G$ induced by $S$. In this section and the next, we will make use of the notion of an $n$-independent set in a multigraph $G$, which is defined as a subset $S$ of $V(G)$ such that $G[S]$ has maximum degree strictly less than $n$. Setting $n=1$ recovers the usual notion of an independent set. Let $\mu_G(xy)$ denote the number of edges between vertices $x$ and $y$ in a multigraph $G$.

If $M$ is the matrix defined in Lemma~\ref{Lemma:setUp} and $G$ is the multigraph whose adjacency matrix agrees with $M$ in its off-diagonal entries, then an $n$-independent set in $G$ corresponds to a principal submatrix of $M$ in which the off-diagonal entries in each row sum to less than $n$. This allows us to use results that guarantee an $n$-independent set in a multigraph to find the diagonally dominant principal submatrix of $M$ that we require. In particular we will use the following result of Caro and Tuza \cite{CarTuz}.

\begin{theorem}[\cite{CarTuz}]\label{Theorem:CT}
Let $n$ be a positive integer and let $G$ be a multigraph. There is an $n$-independent set in $G$ of size at least $\lceil\sum_{u \in V(G)} f_n(\deg_G(u))\rceil$ where
$$f_n(x)=\left\{
         \begin{array}{ll}
           1-\tfrac{x}{2n}, & \hbox{if $x \leq n$;} \\
           \tfrac{n+1}{2(x+1)}, & \hbox{if $x \geq n$.}
         \end{array}
       \right.$$
\end{theorem}

We next prove a technical lemma that enables us to deduce bounds of a specific form that we denote by $CB_{(v,k,\lambda;s)}(\alpha,\beta)$. We will state the bounds in this section and the next in terms of this notation. Observe that the bound of Theorem~\ref{Theorem:main} is $CB_{(v,k,\lambda;s)}(1,0)$.

\begin{lemma}\label{Lemma:CBBound}
Let $s$ and $b_s$ be positive integers and let $\alpha$ and $\beta$ be nonnegative real numbers such that $\alpha \geq 2\beta$. Suppose that any $t$-$(v,k,\lambda)$ covering $(V,\mathcal{B})$ has $b(X) \geq b_s$ for each $X \in \binom{V}{s}$, and $|\mathcal{B}| \geq \alpha|\mathcal{V}_0|+\beta|\mathcal{V}_1|$ where $\mathcal{V}_i=\{X \in \binom{V}{s}:b(X)=b_s+i\}$ for $i \in \{0,1\}$. Then
\[C_{\lambda}(v,k,t) \geq \left\lceil CB_{(v,k,\lambda;s)}(\alpha,\beta) \right\rceil \quad \hbox{where} \quad CB_{(v,k,\lambda;s)}(\alpha,\beta) = \mfrac{b_s(\alpha-\beta)\binom{v}{s}+\alpha \binom{v}{s}}{(\alpha-\beta)\binom{k}{s}+1}.\]
\end{lemma}
\begin{proof}
Let $(V,\mathcal{B})$ be a $t$-$(v,k,\lambda)$ covering. Let $\mathcal{V}=\binom{V}{s}$, $x=|\mathcal{B}|\binom{k}{s}-b_s\binom{v}{s}$ and $v_i=|\mathcal{V}_i|$ for $i \in \{0,1\}$. Note that $v_1+2\left(\binom{v}{s}-v_0-v_1\right) \leq x$ because $b(X) = b_s+i$ for each $X \in \mathcal{V}_i$ for $i \in \{0,1\}$, $b(X) \geq b_s+2$ for each $X \in \mathcal{V} \setminus (\mathcal{V}_0 \cup \mathcal{V}_1)$, and $\sum_{X \in \mathcal{V}}b(X)=|\mathcal{B}|\binom{k}{s}$. It follows that $v_0 \geq \frac{1}{2}(2\binom{v}{s}-v_1-x)$ and so from our hypotheses we have
\[|\mathcal{B}| \geq \tfrac{1}{2}\alpha\left(2\tbinom{v}{s}-v_1-x\right)+\beta v_1 = \alpha\tbinom{v}{s}-\tfrac{1}{2}\alpha x-\tfrac{1}{2}(\alpha-2\beta)v_1.\]
Thus, because $\alpha \geq 2\beta$, it follows from $v_1 \leq |\mathcal{V} \setminus \mathcal{V}_0| \leq x$ that
\[|\mathcal{B}| \geq \alpha\tbinom{v}{s}-\tfrac{1}{2}\alpha x-\tfrac{1}{2}(\alpha-2\beta)x = \alpha \tbinom{v}{s}-(\alpha-\beta)x.\]
Since $x=|\mathcal{B}|\binom{k}{s}-b_s\binom{v}{s}$, we can deduce $|\mathcal{B}| \geq CB_{(v,k,\lambda;s)}(\alpha,\beta)$.
\end{proof}

\begin{remark}\label{Remark:CBboundcomp}
A routine calculation shows that if $b_s+1>\beta\binom{k}{s}$, then the bound $\lceil CB_{(v,k,\lambda;s)}(\alpha,\beta) \rceil$  is inferior to the bound given by $s$ iterated applications of \eqref{equation:Sch} to $b_s+1$.
\end{remark}

\begin{theorem}\label{Theorem:dBig}
Suppose the hypotheses of Lemma~\ref{Lemma:setUp} hold, that $b_s < \binom{k}{s}$, and that $d \geq a_s  \geq 1$. Then
\[C_{\lambda}(v,k,t) \geq \left\lceil
CB_{(v,k,\lambda;s)}\left(\mfrac{a_s+1}{2(d+1)},\mfrac{a_s+1}{2\big(d+\binom{k}{s}\big)}\right)\right\rceil.\]
\end{theorem}

\begin{proof}
Let $(V,\mathcal{B})$ be a $t$-$(v,k,\lambda)$ covering. Let $\mathcal{V}_i=\{X \in \binom{V}{s}:b(X)=b_s+i\}$ for $i \in \{0,1\}$. Let $G$ be the multigraph with vertex set $\binom{V}{s}$ such that $\mu_G(XY)=b(X \cup Y)-b_{|X \cup Y|}$ for each pair of distinct vertices $X$ and $Y$.

By the definition of $G$, for a positive integer $n$, an $n$-independent set $\mathcal{S}$ in the multigraph $G$ is a subset of $\binom{V}{s}$ with the property that, for all $X \in \mathcal{S}$,
\[\medop \sum_{Y \in \mathcal{S} \setminus \{X\}} \left(b(X \cup Y)-b_{|X \cup Y|}\right) < n.\]
Consequently, if $n \leq a_s$, then $\mathcal{S}$ satisfies the hypotheses of Lemma~\ref{Lemma:main} and $|\mathcal{B}|\geq|\mathcal{S}|$. So, by Lemma~\ref{Lemma:CBBound}, it suffices to show that $G$ has an $a_s$-independent set of size at least
\[\mfrac{a_s+1}{2d+2}|\mathcal{V}_0|+\mfrac{a_s+1}{2\big(d+\binom{k}{s}\big)}|\mathcal{V}_1|.\]
By Lemma~\ref{Lemma:setUp}(b), $\deg_G(X)=d$ for all $X \in \mathcal{V}_0$ and $\deg_G(X)=d+\binom{k}{s}-1$ for all $X \in \mathcal{V}_1$. Thus, because $d \geq a_s$, $G$ has an $a_s$-independent set of the required size by Theorem~\ref{Theorem:CT}.
\end{proof}

We only need consider the natural choice of $b_s$ in Theorem~\ref{Theorem:dBig}. This follows by Remark~\ref{Remark:CBboundcomp} because
\[\mfrac{(a_s+1)\binom{k}{s}}{2\big(d+\binom{k}{s}\big)} < \mfrac{a_s+1}{2} < a_s+1 < b_s+1.\]
	
\section{Improved bounds for the case \texorpdfstring{$d<a_s$}{d<as}}\label{impSec}

In this section we will show that, by using techniques similar to those of the last section in the case $d < a_s$, we can sometimes improve on Theorem \ref{Theorem:main}. We require a slight variant of Lemma~\ref{Lemma:main}.

\begin{lemma}\label{Lemma:mainVariant}
Suppose the hypotheses of Lemma~\ref{Lemma:setUp} hold and there exists a subset $\mathcal{S}$ of $\binom{V}{s}$ and positive real numbers $(c_X)_{X \in \mathcal{S}}$ such that, for each $X \in \mathcal{S}$,
\[\medop \sum_{Y \in \mathcal{S} \setminus \{X\}} c_Y\left(b(X \cup Y)-b_{|X \cup Y|}\right) < c_X\left(a_s+b(X)-b_s\right),\]
then $|\mathcal{B}| \geq |\mathcal{S}|$.
\end{lemma}

\begin{proof}
The proof of Lemma~\ref{Lemma:main} applies, except that our hypotheses here imply via the Gershgorin circle theorem (see \cite[p.16-6]{Hog}) that the matrix $M''$ rather than $M'$ is positive definite, where $M''$ is obtained from $M'$ by multiplying the entries in column $X$ by $c_X$ for each $X \in \mathcal{S}$. However, it is easy to see (using Sylvester's criterion \cite[p.9-7]{Hog}, for example) that $M'$ is positive definite if and only if $M''$ is.
\end{proof}

In Section~\ref{genSec} we employed multigraphs, but in this section we will work in a more general setting of edge-weighted graphs. An edge-weighted graph $G$ is a complete (simple) graph in which each edge has been assigned a nonnegative real weight. We denote the weight of an edge $uw$ in such a graph $G$ by $\wt_G(uw)$ and we define the weight of a vertex $u$ of $G$ as $\wt_G(u)=\sum_{w \in V(G) \setminus \{u\}}\wt_G(uw)$. For $S \subseteq V(G)$, let $G[S]$ denote the edge-weighted subgraph of $G$ induced by $S$. We generalise our notion of an $n$-independent set by saying, for a positive integer $n$, that a subset $S$ of the vertices of an edge-weighted graph $G$ is \emph{$n$-independent} in $G$ if $\wt_{G[S]}(u) < n$ for each $u \in S$.

We will require a technical result which guarantees the existence of an $n$-independent set of a certain size in an edge-weighted graph of a specific form. This result was effectively proved in \cite{Hor}.

\begin{lemma}\label{Lemma:weightedIndepSetBound}
Let $n$, $d$ and $d'$ be nonnegative integers such that $d < n < d' -d$, and let $G$ be a multigraph on some vertex set $\mathcal{V}_0 \cup \mathcal{V}_1$ such that $\deg_{G}(X)=d$ for $X \in \mathcal{V}_0$ and $\deg_{G}(X)=d'$ for $X \in \mathcal{V}_1$. Let $c$ be a real number such that $c>\frac{d}{n}$ and let $G^*$ be the edge-weighted graph on vertex set $\mathcal{V}_0 \cup \mathcal{V}_1$ such that, for all distinct $X,Y \in \mathcal{V}_0 \cup \mathcal{V}_1$,
\[\wt_{G^*}(XY)=
\left\{
  \begin{array}{ll}
    0, & \hbox{if $X,Y \in \mathcal{V}_0$;} \\
    \mu_G(XY), & \hbox{if $X,Y \in \mathcal{V}_1$;} \\
    c\mu_G(XY), & \hbox{otherwise.}
  \end{array}
\right.\]
Let $\alpha$ and $\beta$ be real numbers such that one of the following holds.
\begin{itemize}[nosep,topsep=1mm,itemsep=1mm]
	\item[(a)]
$(\alpha,\beta)=\Big(1-\frac{d^2}{2n(n+1)},\frac{n+2}{2(d'+1)}\Big)$.
	\item[(b)]
$(\alpha,\beta)=\Big(1,1-\frac{dd'}{n(n+1)}\Big)$, $d\geqslant \frac{n}{2}$ and $dd' < n(n+1)$.
	\item[(c)]
$(\alpha,\beta)=\Big(1, \sqrt{\frac{d(n+2)}{(n+1)(n-d)}}-\frac{d(d'+1)}{2(n+1)(n-d)}\Big)$, $d< \frac{n}{2}$, and $d(d'+1)^2 < 4(n+1)(n+2)(n-d)$.
\end{itemize}
Then $\alpha \geq 2\beta > 0$ and, if $c$ is sufficiently close to $\frac{d}{n}$, $G^*$ has an $(n+1)$-independent set $\mathcal{S}$ such that $\mathcal{V}_0 \subseteq \mathcal{S}$ and $|\mathcal{S}| \geq \alpha|\mathcal{V}_0| +\beta|\mathcal{V}_1|$.
\end{lemma}

\begin{proof}
When (a) holds we obviously have $\beta>0$ and
\[\alpha-2\beta=\mfrac{(d'-n-d-1)(2n(n+1)-d^2)+(n-d)(2d(n+1)+d^2)}{2n(n+1)(d'+1)}\]
is nonnegative because $d'>n+d$ and $n>d$. When (b) holds we have $\beta>0$ because $dd' < n(n+1)$ and
\[\alpha-2\beta=\mfrac{2dd'-n(n+1)}{n(n+1)}\]
is nonnegative because $d'>n+d$ and $d \geq \frac{n}{2}$. When (c) holds we have $\beta>0$ because $d(d'+1)^2 < 4(n+1)(n+2)(n-d)$ and $\frac{d(d'+1)}{(n+1)(n-d)} \geq \frac{d(n+2)}{(n+1)(n-d)}$ because $d'>n$. Thus, since $2\sqrt{x}-x \leq 1$ for each nonnegative real number $x$, we have $\alpha \geq 2\beta$.

In the course of the proof of \cite[Theorem 14]{Hor}, the remainder of this result is proved for the case $d'=d+k-1$. It is a routine exercise to show that the proof given there applies here for any $d' > n+d$.
\end{proof}

We can now establish our improvements on Theorem~\ref{Theorem:main}.

\begin{theorem}\label{Theorem:smallDImprovements}
Suppose the hypotheses of Lemma~\ref{Lemma:setUp} hold, that $b_s < \binom{k}{s}$, and that $d < a_s$. Let $d'=d+\binom{k}{s}-1$. Then $C_\lambda(v,k,t) \geq \left\lceil CB_{(v,k,\lambda;s)}\left(\alpha,\beta\right)\right\rceil$ when one of the following holds.
\begin{itemize}
	\item[(a)]
$(\alpha,\beta)=\Big(1-\frac{d^2}{2a_s(a_s+1)},\frac{a_s+2}{2(d'+1)}\Big)$.
	\item[(b)]
$(\alpha,\beta)=\Big(1,1-\frac{dd'}{a_s(a_s+1)}\Big)$, $d \geq \frac{a_s}{2}$ and $dd' < a_s(a_s+1)$.
	\item[(c)]
$(\alpha,\beta)=\Big(1, \sqrt{\frac{d(a_s+2)}{(a_s+1)(a_s-d)}}-\frac{d(d'+1)}{2(a_s+1)(a_s-d)}\Big)$, $d < \frac{a_s}{2}$ and $d(d'+1)^2 < 4(a_s+1)(a_s+2)(a_s-d)$.
\end{itemize}
\end{theorem}

\begin{proof}
Let $(V,\mathcal{B})$ be a $t$-$(v,k,\lambda)$ covering. Let $\mathcal{V}_i=\{X \in \binom{V}{s}:b(X)=b_s+i\}$ for $i \in \{0,1\}$. Let $G$ be the multigraph with vertex set $\binom{V}{s}$ such that $\mu_G(XY)=b(X \cup Y)-b_{|X \cup Y|}$ for each pair of distinct vertices $X$ and $Y$. Note that, by Lemma~\ref{Lemma:setUp}, $\deg_G(X)=d$ for each $X \in \mathcal{V}_0$ and $\deg_G(X)=d'$ for each $X \in \mathcal{V}_1$. Also, $d<a_s<d'-d$ because $d'-d=\binom{k}{s}-1$ and $a_s < b_s < \binom{k}{s}$. Thus, by Lemma~\ref{Lemma:weightedIndepSetBound}, there is a real number $c > \frac{d}{a_s}$ such that the edge-weighted graph $G^*$ obtained from $G[\mathcal{V}_0 \cup \mathcal{V}_1]$ as in Lemma~\ref{Lemma:weightedIndepSetBound} has an $(a_s+1)$-independent set $\mathcal{S}$ such that $\mathcal{V}_0 \subseteq \mathcal{S}$ and $|\mathcal{S}| \geq \alpha|\mathcal{V}_0| +\beta|\mathcal{V}_1|$.
We show that we can apply Lemma~\ref{Lemma:mainVariant} to $\mathcal{S}$ choosing $c_X=c$ for $X \in \mathcal{S} \cap \mathcal{V}_0$ and $c_X=1$ for $X \in \mathcal{S} \cap \mathcal{V}_1$. By Lemma~\ref{Lemma:CBBound} this will suffice to complete the proof.

If $X \in \mathcal{S} \cap \mathcal{V}_0$, then $c_X=c$, $b(X)=b_s$, and
\[\medop \sum_{Y \in \mathcal{S} \setminus \{X\}} c_Y\left(b(X \cup Y)-b_{|X \cup Y|}\right) \leq d < ca_s=c_X\left(a_s+b(X)-b_s\right)\]
where the first inequality follows from Lemma~\ref{Lemma:setUp}(b).
If $X \in \mathcal{S} \cap \mathcal{V}_1$, then $c_X=1$, $b(X)=b_s+1$, and
\[\medop \sum_{Y \in \mathcal{S} \setminus \{X\}} c_Y\left(b(X \cup Y)-b_{|X \cup Y|}\right)
= \wt_{G^*[\mathcal{S}]}(X) < a_s+1 = c_X\left(a_s+b(X)-b_s\right)\]
where the first equality follows from the definition of $G^*$ and our choice of $c_Y$ for $Y \in \mathcal{S}$ and the inequality follows from the fact that $\mathcal{S}$ is an $(a_s+1)$-independent set in $G^*$.
\end{proof}

Again, we only need consider the natural choice of $b_s$ in Theorem~\ref{Theorem:smallDImprovements}. To establish this it suffices, by Remark~\ref{Remark:CBboundcomp} and the fact that $b_s>a_s$, to show that $a_s+2-\beta\binom{k}{s}$ is positive.  When (a) holds this is the case because
\[\mfrac{(a_s+2)\binom{k}{s}}{2(d'+1)} \leq \mfrac{a_s+2}{2} < a_s+2.\]
When (b) or (c) holds, $a_s+2-\beta\binom{k}{s}$ is a quadratic in $\binom{k}{s}$ (note that $d'=d+\binom{k}{s}-1$) and we can compute its global minimum in terms of $a_s$ and $d$. When (b) holds this minimum is equal to
\[\mfrac{1}{4da_s(a_s+1)}\left((2d-a_s)(a_s^3 + 2a_s^2 + ad + a_s) + d(2a_s^3 - d^3 + 7a_s^2 + 4a_s) +
  d^2(2a_s^2 + 2d - 1)\right)\]
which is positive since $\frac{a_s}{2} \leq d <a$. When (c) holds this minimum is equal to
\begin{multline*}
\mfrac{1}{8(a_s+1)(a_s-d)}\big(4d\sqrt{d(a_s+1)(a_s+2)(a_s-d)} + (a_s - 2d)(2a_s^2 + 6a_s + 12) \\
+ (2a_s^3 - d^3 +16d) + 2a_s(3a_s-2)\big)
\end{multline*}
which is positive since $0\leq d<\frac{a_s}{2}$.

There are situations in which each of the Theorem~\ref{Theorem:smallDImprovements} bounds is superior to both of the others. In the special case when $d=0$, Theorem~\ref{Theorem:smallDImprovements}(a) is the best of our bounds.

\section{Improvements for small parameter sets}

We conclude with some tables which detail small parameter sets for which the results in this paper produce an improvement over the previously best known lower bound on $C(v,k,t)$. For $t=2$ similar tables appear in \cite{Hor}, so we concentrate here on the case $t \geq 3$. Our methodology in producing these tables is as follows.

To determine whether we see an improvement for $C(v',k',t')$ we successively evaluate a ``best known'' bound $b_{(v,k,t)}$ for $C(v,k,t)$ for $(v,k,t)=(v'-t'+1,k'-t'+1,1),(v'-t'+2,k'-t'+2,2),\ldots,(v',k',t')$. This ``best known'' bound incorporates the following.
\begin{itemize}[itemsep=0.5mm,parsep=0mm]
    \item
$C(v,k,1)=\lceil \frac{v}{k} \rceil$.
    \item
$C(v,k,t) \geq \lceil \frac{v}{k}b_{(v-1,k-1,t-1)} \rceil$ by \eqref{equation:Sch}.
    \item
The Mills and Mullin result stated in \eqref{equation:MilMul}.
    \item
Results for a fixed number of blocks from \cite{Mil,GreLiVan,Tod85,TodTon}. These include results for $t=2$, for $t=3$, and for general $t$. (The $t \in \{2,3\}$ results are summarised in \cite{GorSti}.)
    \item
Theorems 2.1, 3.1 and 4.4 of \cite{TodLB}.
    \item
The lower bound of de Caen \cite{deC}.
    \item
The lower bounds listed for $t \leq 8$, $v \leq 99$, $k \leq 25$ at the La Jolla Covering Repository \cite{Gor}.
    \item
Theorems~\ref{Theorem:main},~\ref{Theorem:dBig} and~\ref{Theorem:smallDImprovements} of this paper, applied with $s \in \{1,\ldots,\lfloor\frac{t}{2}\rfloor\}$ and with $b_i$ chosen as $b_{(v-i,k-i,t-i)}$ for $i \in \{s,\ldots,2s\}$ (note that these theorems with $s=1$ specialise to the results in \cite{Hor}).
\end{itemize}
If the bound provided for $C(v',k',t')$ by one of the theorems of this paper (using a particular choice of $s$) strictly exceeds the bound provided by any of the other results, then we include $v'$ in the appropriate location in the tables. If, moreover, the bound provided for $C(v',k',t')$ by Theorem~\ref{Theorem:dBig} or Theorem~\ref{Theorem:smallDImprovements} strictly exceeds the bound provided by Theorem~\ref{Theorem:main}, then the table entry is set in italic or bold font, respectively. All improvements for $k \leq 40$ when $t=3$, when $t\in \{4,5\}$ and when $t\in\{6,7,8\}$ are given in Tables~\ref{table:t3}, \ref{table:t45}, and \ref{table:t678} respectively (recall from the discussion after Theorem~\ref{Theorem:main} that we obtain no improvements for sufficiently large $v$). Of course the listed improvements will, via \eqref{equation:Sch}, imply many further improvements for higher values of $t$, but we do not include these subsequent improvements in our tables.

\begin{table}[htbp]
	\centering
	\footnotesize
\caption{$v$'s with an improved lower bound on $C(v,k,t)$ when $t=3$}\label{table:t3}
	\begin{tabular}{|l|l|}
		\hline
		$k$ & $s=1$\\\hline
		9&19 \\
		10&21,\textbf{22} \\
		12&\textbf{26} \\
		13&29 \\
		15&33,42,45 \\
		16&35,\textbf{36},45,46,48,49 \\
		17&\textbf{33},48,49,51,52,53 \\
		18&\textbf{35},40,51,59 \\
		19&\textbf{37},42,43,54,\textbf{55},58,62 \\
		20&\textbf{39},\textit{44},57,61,62,66 \\
		21&41,47,60,61,64,65,\textbf{66},69 \\
		22&43,49,50,63,64,73,88,89 \\
		23&45,\textit{51},66,71,76,87,\textbf{88},\textbf{89},92,93,95,\textbf{96},\textbf{97} \\
		24&47,\textit{53},54,69,74,\textbf{75},80,91,92,93,96,97,99,\textbf{101} \\
		25&49,56,57,\textbf{72},73,77,\textbf{78},79,83,95,96,97,100,101 \\
		26&51,\textit{58},75,87,\textbf{100},101,104,105,106 \\
		27&53,\textit{60},61,78,84,90,103,\textbf{104},\textbf{105},108,109,110,114,115 \\
		28&55,\textit{62},63,64,81,\textbf{82},87,88,94,107,108,\textbf{109},112,113,114,117,118,119 \\
		29&57,\textit{64},\textit{65},84,\textbf{85},90,91,\textbf{92},97,111,112,113,116,117,118,121,122,123,124 \\
		30&59,\textit{67},68,87,101,115,116,117,120,121,122,126,127,\textbf{128} \\
		31&61,\textit{69},70,71,90,\textbf{91},97,104,119,\textbf{120},121,124,125,126,127,130,131,\textbf{132},\textbf{133} \\
		32&63,\textit{71},\textit{72},93,100,\textbf{101},\textit{107},108,123,\textbf{124},\textbf{125},129,130,131,135,136,137,160,161 \\
		
		33&65,\textit{73},\textit{74},75,96,97,103,\textbf{104},105,111,127,128,\textbf{129},133,134,135,139,140,141, \\
		
		&158,\textbf{159},\textbf{160},\textbf{161},165,166,168,169,\textbf{170},\textbf{171} \\

		34&67,\textit{76},77,78,99,100,\textit{106},\textit{114},115,131,132,\textbf{133},137,138,139,143,144,145,\textbf{146},\\
		
		&163,164,165,166,170,171,173,174,\textbf{175},\textbf{176},\textbf{177} \\
		
		35&69,\textit{78},\textit{79},102,\textit{109},110,\textit{117},118,135,136,\textbf{137},141,142,143,148,149,\textbf{150},168,169,\\
		
		&170,171,175,176,179,\textbf{180},\textbf{181},\textbf{182} \\
		
		36&71,\textit{80},\textit{81},82,\textbf{105},113,114,122,139,140,141,145,146,147,\textbf{148},152,153,\textbf{154},155,\\
		
		&174,175,176,180,181,184,\textbf{186},\textbf{187} \\
		
		37&73,\textit{82},\textit{83},84,\textbf{85},\textbf{108},109,116,117,118,\textit{124},125,143,144,145,150,151,\textbf{152},157,\\
		
		&158,159,178,179,\textbf{180},\textbf{181},183,185,186,187,189,\textbf{192},\textbf{193},195,196 \\
		
		38&75,\textit{85},\textit{86},\textbf{111},\textit{119},\textit{128},129,147,\textbf{148},149,154,155,\textbf{156},161,162,163,183,184,185,\\
		&\textbf{186},189,190,\textbf{191},192,\textbf{197},198,201 \\
		
		39&77,\textit{87},\textit{88},89,114,\textit{122},123,132,151,\textbf{152},153,158,159,\textbf{160},165,166,167,\textbf{168},188,\\
		
		&189,190,191,194,195,\textbf{196},\textbf{197},201,202,203,206 \\
		
		40&79,\textit{89},\textit{90},91,\textbf{92},117,118,\textit{125},126,127,\textit{134},\textit{135},136,155,\textbf{156},157,162,163,\textbf{164},\\
		&170,\textbf{171},\textbf{172},193,194,195,196,199,200,\textbf{201},\textbf{202},205,206,207,208,209,212 \\\hline
	\end{tabular}
\end{table}%

\begin{table}[htbp]
	\centering
    \caption{$v$'s with an improved lower bound on $C(v,k,t)$ when $t=4,5$}\label{table:t45}
	\footnotesize
	\begin{tabular}{|l||l||l|l||}
		\hline
		& \multicolumn{1}{|c||}{$t=4$}&\multicolumn{2}{|c||}{$t=5$}\\\hline
		$k$ & $s=1$& $s=1$& $s=2$\\
		\hline
		9&&&17\\
		11&&&\textbf{29} \\
		14&&&\textbf{47} \\
		15&&&42\\
		16&33& &55\\
		17&\textbf{30},35  & &\textbf{59} \\
		18&32,37  & &\textbf{66} \\
		19&34,39  & &70\\
		20&\textbf{39},41  &&\\
		21&37,\textbf{41},43  & &75,93 \\
		22&39,43,45,\textbf{46}  &36&\textbf{79},\textbf{98} \\
		23&41,45,\textbf{48},52  & &\textbf{87},123 \\
		24&43,47,\textbf{50}  &&\\
		25&\textbf{37},45,49,\textbf{52},59  &41&113,\textbf{135},141 \\
		26&51,\textbf{54}  & &\textbf{118} \\
		27&\textit{47},48,53  &44&\textbf{127},147 \\
		28&50,55,64,66,\textbf{68},70  &46,52 &132,\textbf{153} \\
		29&43,52,57,61,\textbf{69}  &54& \\
		30&\textbf{54},59,63,73,75,76  &49,54,56 &\textit{138},\textit{147},161,\textbf{192} \\
		31&\textit{54},\textbf{56},61,65,71,73,74,\textbf{80}  &51,56 &143,171,199,\textbf{206} \\
		32&\textit{56},\textit{57},63,67,74,78,81  &65&\textbf{148},\textbf{177},\textbf{206},213 \\
		33&49,59,65,69,76,78,79,80,\textbf{81},\textbf{85},88  &54,67 &\textbf{158} \\
		34&61,67,71,81,86  &56,61,69 &216\\
		35&\textit{61},\textbf{63},69,73,81,83,84,85,86,\textbf{90},93  &63,\textbf{66},71 &\textbf{227},231,235,\textbf{259} \\
		36&\textit{63},65,71,75,\textbf{76},83,86,91,92,93,96,97  &59,65,68,73 &201\\
		37&55,\textit{65},\textit{66},67,73,\textbf{78},88,\textbf{89},90,91  &61,67,70,75 &\textbf{207},275 \\
		38&\textit{68},\textbf{75},80,88,91,\textbf{93},96,97,101,105  &77&\textbf{218},248,\textbf{283} \\
		39&\textit{68},70,\textbf{77},82,90,93,95,96,99,104  &64,70,79 &224,\textbf{255},264,\textbf{287},\textbf{299} \\
		40&\textit{70},72,\textbf{79},84,95,\textbf{96},98,101,102,103,106,\textbf{107},113,114 &66,72,81 &\textit{230},299,307 	\\
		
		\hline
		
	\end{tabular}
\end{table}%

\begin{table}[htbp]
	\centering
    \caption{$v$'s with an improved lower bound on $C(v,k,t)$ when $t=6,7,8$}\label{table:t678}
	\footnotesize
	\begin{tabular}{|l||l|l|l||l|l|l||l|l|l|}
		\hline
		& \multicolumn{3}{|c||}{$t=6$}& \multicolumn{3}{|c||}{$t=7$}& \multicolumn{3}{|c|}{$t=8$}\\\hline
		$k$ & $s=1$& $s=2$& $s=3$& $s=1$& $s=2$& $s=3$& $s=1$& $s=2$& $s=3$\\\hline
		9& & &\textbf{25} &&&&&&\\
		12& &23& &&&&&&\\
		16& &\textbf{29} & &&&&&&\\
		17& &33,38 & & &31& &&&\\
		18& &43& &&&&&&\\
		19&&&& & &\textbf{75} &&&\\
		20&&&& &\textbf{39} & &&&\\
		21&30&57& &&&&&&\\
		22&33& & &&&&&&\\
		23& &58& & &\textbf{51},53 & &&&\\
		24&36& & &33& & &&&\\
		25& &\textit{58} & & & &\textbf{125} &&&\\
		26&\textbf{36},39 & & &&&&&&\\
		27& &\textbf{63},75 & &&&& & &\textbf{97}  \\
		28& &68,78 & & &68&166&&&\\
		29& &\textbf{83},94 & &40&68& &&&\\
		30& &86& & &\textit{57} & &&&\\
		31&43,49 &\textit{89} & &&&&41& &  \\
		32& &\textit{82} & & &61,77 &221&&&\\
		33&52,55 &\textbf{117},127 & &45&65,82,85 & & &63&  \\
		34&47& & & &86,92 & &45& &  \\
		35&55,58 &91,143& &&&&&&\\
		36&50,57,60 &124& & &\textbf{102},\textbf{105},107 & & &71&  \\
		37& &\textbf{122},139,156 & &&&&49& &170,181\\
		38&60,63 &\textit{108},\textbf{122},136,\textbf{143} & &52&102& &\textbf{52} & &  \\
		39&54,65 &\textit{111} & &\textbf{58} & &&&&\\
		40&61,63 &114,179& 	&&&&53&96& \\\hline
	\end{tabular}
\end{table}%

\medskip

\noindent \textbf{Acknowledgements:} Thanks to Peter Dukes and Vedran Kr\v{c}adinac for the discussions at the \emph{Combinatorics 2016} conference that led to this research. The first author was supported by Australian Research Council grants DE120100040 and DP150100506.


\begin{thebibliography}{99}
    \bibitem{BluGreHei}
I. Bluskov, M. Greig and M.K. Heinrich, Infinite classes of covering numbers, {\it
Canad. Math. Bull.} {\bf 43} (2000), 385--396.

    \bibitem{Bos}
R.C. Bose, A Note on Fisher's Inequality for Balanced Incomplete Block Designs, {\it Ann. Math.
Statistics} {\bf 20} (1949), 619–-620.

    \bibitem{BosCon}
R.C. Bose and W.S. Connor, Combinatorial properties of group divisible incomplete block designs,
{\it Ann. Math. Stat.} {\bf 23} (1952), 367--383.

    \bibitem{BryBucHorMaeSch}
D. Bryant, M. Buchanan, D. Horsley, B. Maenhaut and V. Scharaschkin, On the non-existence of pair
covering designs with at least as many points as blocks, {\it Combinatorica} {\bf 31} (2011),
507--528.

    \bibitem{CarTuz}
Y. Caro and Z. Tuza, Improved lower bounds on $k$-independence, {\it J. Graph Theory} {\bf 15} (1991), 99--107.

    \bibitem{deC}
D. de Caen, Extension of a theorem of Moon and Moser on complete subgraphs, {\it Ars Combin.} {\bf 16} (1983), 5--10.


    \bibitem{ErdHan}
P. Erd{\H{o}}s and H. Hanani, On a limit theorem in combinatorial analysis, {\it Publ. Math.
Debrecen} {\bf 10} (1963), 10--13.

    \bibitem{ErdRen}
P. Erd{\H{o}}s and A. R{\'e}nyi, On some combinatorical problems, {\it Publ. Math.
Debrecen} {\bf 4} (1956), 398--405.

    \bibitem{Fis}
R.A. Fisher, An examination of the different possible solutions of a problem in incomplete blocks,
{\it Ann. Eugenics} {\bf 10} (1940), 52--75.



    \bibitem{Fur}
Z. F{\"u}redi, Covering pairs by {$q^2+q+1$} sets, {\it J. Combin. Theory Ser. A} {\bf 54} (1990), 248--271.

    \bibitem{GloKuhLoOst}
S. Glock, D. K\"uhn, A. Lo and D. Osthus, The existence of designs via iterative absorption, arXiv:1611.06827.

    \bibitem{Gor}
D.M. Gordon,
La Jolla Covering Repository,
{\tt http://www.ccrwest.org/cover.html}.

	\bibitem{gordon95}
D.M. Gordon, O. Patashnik and G. Kuperberg, New constructions for covering designs, \textit{J. Combin. Des.} \textbf{3} (1995), 269--284.

    \bibitem{GorSti}
D.M. Gordon and D.R. Stinson, Coverings, in {\it The CRC Handbook of Combinatorial Designs, {\rm 2}nd edition} (Eds. C.J. Colbourn, J.H. Dinitz), CRC Press (2007), 365--373.

    \bibitem{GreLiVan}
M. Greig, P.C. Li and G.H.J. van Rees, Covering designs on 13 blocks revisited, {\it Util. Math.} {\bf 70} (2006), 221--261.

    \bibitem{Hog}
L. Hogben, ed. {\it Handbook of linear algebra, 2nd edition}, Chapman and Hall/CRC (2013).

	\bibitem{Hor} D. Horsley, Generalising Fisher's inequality to coverings and packings arXiv:1409.0485 (2015).

    \bibitem{Kee}
P. Keevash, The existence of designs, arXiv:1401.3665.

    \bibitem{Mil}
W.H. Mills, Covering designs. I. Coverings by a small number of subsets, {\it Ars Combin.} {\bf 8} (1979), 199--315.

    \bibitem{MilMul}
W.H. Mills and R.C. Mullin, Coverings and packings, in {\it Contemporary Design Theory}, (Eds. J.H. Dinitz and D.R. Stinson), Wiley, (1992), 371--399.

	\bibitem{raywil75}
D.K. Ray-Chaudhuri and R.M. Wilson, On $t$-designs, {\it Osaka J. Math.} {\bf 12} (1975), 737--744.

    \bibitem{Rod}
V. R\"{o}dl, On a packing and covering problem, {\it European J. Combin.} {\bf 6} (1985), 69--78.

    \bibitem{Sch}
J. Sch{\"o}nheim, On coverings, {\it Pacific J. Math.} {\bf 14} (1964), 1405--1411.

    \bibitem{TodFP}
D.T. Todorov, Some coverings derived from finite planes, {\it Colloq. Math. Soc. J\'anos Bolyai}
{\bf 37} (1984), 697--710.

    \bibitem{Tod85}
D.T. Todorov, On some covering designs, {\it J. Combin. Theory Ser. A} {\bf 39} (1985), 83--101.

    \bibitem{TodLB}
D.T. Todorov, Lower bounds for coverings of pairs by large blocks, {\it Combinatorica} {\bf 9}
(1989), 217--225.

    \bibitem{TodTon}
D.T. Todorov and V.D. Tonchev, On some coverings of triples, {\it C. R. Acad. Bulgare Sci.} {\bf 35} (1982), 1209--1211.

	\bibitem{wil82}
R.M. Wilson, Incidence matrices of $t$-designs, {\it Linear Algebra Appl.} {\bf 46}  (1982), 73--82.

\end{thebibliography}
\end{document}